\documentclass[final]{siamart190516}

\usepackage{amsmath,amssymb,amsfonts}
\usepackage{mathtools}
\usepackage{enumerate} 

\usepackage{pgfplots}
\usepackage{hyperref}

\DeclarePairedDelimiter{\norm}{\lVert}{\rVert}
\DeclareMathOperator{\adj}{adj}
\DeclareMathOperator{\tr}{tr}

\newcommand{\R}{\mathbb{R}}

\newcommand{\coeff}{\nu}

\newcommand{\cratio}{\chi}

\newsiamremark{remark}{Remark}
\newsiamremark{example}{Example}

\hypersetup{%
    pdftitle={Structured backward errors in linearizations},
    pdfauthor={Vanni Noferini, Leonardo Robol, Raf Vandebril},
    pdfkeywords={a,b}
}

\title{Structured backward errors in linearizations\thanks{
The work of Vanni Noferini was supported by an Academy of Finland grant (Suomen Akatemian päätös 331240); The work 
of Leonardo Robol was supported by an INdAM/GNCS research grant ``Metodi low-rank per problemi di algebra lineare con
struttura data-sparse''. 
}}

\author{Vanni Noferini\thanks{Department of Mathematics and Systems Analysis
, Aalto University, PL 11000, 00076 Aalto,
Finland. Email address: \texttt{vanni.noferini@aalto.fi}.}
        \and Leonardo Robol\thanks{Department of Mathematics, University of Pisa, Largo Bruno Pontecorvo 5, 56127 Pisa, Italy. Email address: \texttt{leonardo.robol@unipi.it}. The author is a member of the INdAM research group GNCS.} 
		\textsuperscript{,}\thanks{Institute of Information Science and
		Technologies ``A. Faedo'', ISTI-CNR, Pisa, Italy. \texttt{leonardo.robol@isti.cnr.it}.} \and 
		Raf Vandebril\thanks{Department of Computer Science, KU Leuven, Celestijnenlaan 200A, 3001 Heverlee, Belgium. Email address: \texttt{raf.vandebril@cs.kuleuven.be}.}}

\headers{STRUCTURED BACKWARD ERRORS IN LINEARIZATIONS}{V.~NOFERINI, L.~ROBOL AND R.~VANDEBRIL}

\begin{document}

\maketitle

\begin{abstract}
	A standard approach to compute the roots of a univariate polynomial is to
	compute the eigenvalues of an associated \emph{confederate} matrix instead, such
	as, for instance the companion or comrade matrix. The eigenvalues of the
	confederate matrix can be computed by Francis's QR algorithm. Unfortunately,
	even though the QR algorithm is provably backward stable, mapping the errors
	back to the original polynomial coefficients can still lead to huge errors.
	However, the latter statement assumes the use of a non-structure exploiting QR
	algorithm. In [J. Aurentz et al., Fast and backward stable computation of roots of polynomials, SIAM J. Matrix Anal. Appl., 36(3), 2015] it was shown that a structure exploiting QR algorithm for
	companion matrices leads to a structured backward error on the companion matrix.
	The proof relied on decomposing the error into two parts: a part related to the
	recurrence coefficients of the basis (monomial basis in that case) and a part
	linked to the coefficients of the original polynomial. 
	In this article we  prove that the analysis can be extended
	to other classes of comrade matrices. We first provide an alternative
	backward stability proof in the monomial basis using 
	structured QR algorithms; our new point of view shows more explicitly how a structured,
	decoupled error on the confederate matrix gets mapped to the associated
	polynomial coefficients. This insight reveals which properties must be preserved
	by a structure exploiting QR algorithm to end up with a backward stable
	algorithm. We will show that the previously formulated companion analysis fits
	in this framework and we will analyze in more detail Jacobi polynomials (Comrade
	matrices) and Chebyshev polynomials (Colleague matrices).
\end{abstract}



\section{Introduction}
A standard approach to find the solutions of a univariate polynomial equation is to convert the problem into an equivalent one where the eigenvalues of a matrix are computed instead. The algebraic technique used to construct such a matrix is called a linearization and, albeit ultracentenarian, 
it is still the most popular initial step of modern rootfinding algorithms, at least if all the polynomial roots are sought. For example, this is what MATLAB's {\tt roots} function does \cite{edelman1995polynomial} for polynomials expressed in the monomial basis, and it is at the heart of {\tt chebfun/roots} for polynomials expressed in the Chebyshev basis \cite{chebfunguide}.

In the landmark paper \cite{edelman1995polynomial} Edelman and Murakami cast a shadow on this strategy. They showed that, even if the matrix eigenvalue problem is solved with a backward stable algorithm, such as QR \cite{b333}, the whole approach can (depending on the specific linearized polynomial) be catastrophically unstable.  More recently, De Ter\'{a}n, Dopico and P\'{e}rez \cite{ddp} argued  that, if Fiedler linearizations \cite{fiedler} are used instead of the classical companion linearization \cite{edelman1995polynomial}, the potential misfortunes can be even more pronounced. Fortunately, Van Dooren and Dewilde \cite{vandoorendewilde} showed that this problem could be circumvented by solving a generalized eigenproblem instead; the disadvantage, however, is that this is significantly less efficient.

While De Ter\'an, Dopico, and P\'erez \cite{ddp} and Edelman and Murakami \cite{edelman1995polynomial} focused only on polynomials expressed in the monomial basis, Nakatsukasa and Noferini \cite{nn16} proved that analogous results on the dangers of always trusting the linearize-and-use-QR philosophy can be stated for any degree-graded basis. That is, the beautiful idea of constructing the so-called \emph{confederate} matrix and then finding its eigenvalues by QR is potentially, depending on the polynomial input, unstable. Again, for many bases of practical importance, switching to a pencil and the QZ algorithm provably avoids any instabilities \cite{lawrence14, lawrence16, nn16}.
Less clear than in the monomial basis is under which conditions on the input polynomial the QR-based approach is stable; Noferini and P\'{e}rez \cite{np17} gave a complete answer for the Chebyshev basis, but we are not aware of any progress for other bases.

This story has recently seen a sudden twist towards positive news.
All the aforementioned results rely on the assumption that a \emph{general} eigensolver, e.g.,
{\em unstructured}
QR, is applied to the linearizing matrix.
However, confederate matrices are typically highly structured. Algorithms specifically designed to preserve and utilize this structure result in two advantages: reduced computational and storage costs, and a structured backward error. 
Several such algorithms can be found in the literature: see, for example, \cite{aurentz2017fast,aurentz2015fast,gemignani2017fast} for companion matrices and the
references therein for the case of unitary (fellow) and symmetric plus low rank (comrade) matrices.

Consider, for instance, the companion algorithm presented by Aurentz et al.\
\cite{aurentz2015fast}. There the authors
proved that the structured QR algorithm has a  backward error on the companion 
matrix of the order of $\norm{p}_2^2  \epsilon_m$
for the rank one part, and of the order of $\epsilon_m$ for the unitary part (with $\epsilon_m$ denoting the machine precision).
This implies that \emph{as an eigensolver} that particular algorithm is not stable, and a blind application of the results of Edelman and Murakami \cite{edelman1995polynomial},
merging both errors, would yield a backward error on the polynomial of the size
$\norm{p}_2^3  \epsilon_m$: an apparent disaster, as this is even worse than what the
unstructured QR obtains: $\norm{p}_2^2 \epsilon_m$. In the numerical experiments, however, only an error of the form $\norm{p}_2^2 \epsilon_m$ was observed, insinuating that something peculiar was happening with the errors.

Two years later, Aurentz et al.\ \cite{aurentz2018core,aurentz2017fast}, were able to improve their
companion code to get an error of the order of $\norm{p}_2  \epsilon_m$
for the rank one part. According to the results of Edelman and Murakami, this should have
implied an error of size about $\norm{p}_2^2$ on the polynomial coefficients. 
However, by considering a mixed backward error analysis they demonstrated that the specific structure of the backward error on the companion matrix implies that \emph{as a rootfinder}, considering the backward error on the polynomial, the algorithm  is backward stable, with a backward error bounded by $\norm{p}_2\epsilon_m$!
This was the first time that a rootfinder based on linearization and (structured) QR was proved to be stable in this stronger sense.

In the current paper we extend the backward error results of Aurentz et al.\
\cite{aurentz2017fast} to other confederate matrices. As a first step, we present
an alternative derivation of the same result of \cite{aurentz2017fast}, which is
less coupled with the underlying algorithm  
and thus easier to generalize to other bases.
We examine how to cleverly map the
structured backward error on the confederate matrix back to the polynomial coefficients.
As an example of particular interest
we analyze the case of Chebyshev polynomials (colleague matrices) in
detail, see how the companion results \cite{aurentz2017fast} fit in, 
and later discuss the extension to more general Jacobi
polynomials (comrade matrices). 

More specifically, we address the following problem. We assume we are given a confederate
pencil, that is, a structured plus rank one pencil that linearizes a polynomial $p$
expressed in a degree graded basis. The pencil is of the form $M(x) +
ab^T$, where $M(x)$ is independent of $p$ and links to the polynomial basis, and the
rank-one addend $ab^T$ encodes the coefficients of $p$.  This is precisely
the scenario encountered
for polynomials expressed in a broad class of orthogonal polynomial bases, including
monomials, Chebyshev, Legendre, ultraspherical, and other Jacobi polynomials.  Next, we
assume that a \emph{structured} eigensolver is used to compute the eigenvalues of the
structured pencil, such that backward errors of different form can be attached to $M(x)$
and to $a b^T$. The question of interest is to map the error back to $p$ and to
characterize it, thereby assessing the overall stability of the rootfinding algorithm.  We
show that under minimal assumptions on the pencil $M(x)$ (satisfied in practice by most
linearization schemes), only the backward error on $M(x)$ increases when mapping it back
to the polynomial.

Our result thus clarifies the direction that should be followed in the development
of stable structured QR algorithms for polynomial rootfinding: one has to ensure
that the backward error on the ``basis part'' of the pencil, the addend $M(x)$,
is small, and independent of the polynomial under consideration. 

The paper is structured as follows. In Section~\ref{sec:conf} we introduce
confederate matrices, prove properties essential for the article and refine to
comrade matrices. Section~\ref{sec:mberr} discusses the basic
principles for the mixed backward error analysis; it is shown how the structured error can be
mapped back to the polynomials.
In Section~\ref{sec:monomial} we illustrate the main idea by reconsidering the
companion matrix, and providing an alternative and simpler derivation of the results of
Aurentz et al.\ \cite{aurentz2015fast}. 
In Section~\ref{sec:comrade} we provide specific bounds for polynomials in the Chebyshev
basis (colleague matrices) and come up with a conjecture for Jacobi polynomials.
We conclude with Section~\ref{sec:conc}.

\section{Confederate matrices}
\label{sec:conf}

First we discuss general confederate matrices. Then we refine to  companion and comrade
matrices, and discuss the special case of colleague matrices.

\subsection{Definition and properties of confederate matrices}
Let ${\phi_j}$ be any degree-graded (i.e., $\deg \phi_j = j$) polynomial basis, such that, for all $j=0,\dots,n$, $\phi_j$ has leading coefficient $\coeff_j \neq 0$ when expressed in the monomial basis. Let $p$ be a polynomial of degree $n$, monic in the basis $\{\phi_j\}$. Denoting
$$ \Phi(x) = \begin{bmatrix}
\phi_{n-1}(x)\\
\vdots\\
\phi_1(x)\\
\phi_0(x)
\end{bmatrix}$$ 
we can write $p(x) = \phi_n(x) + c^T \Phi(x)$ for a unique coefficients vector $c$. Following \cite{barnett, nn16} we now introduce the confederate matrix of $p(x)$.

\begin{definition} \label{def:confederate}
	The confederate matrix of $p(x) = \phi_n(x) + c^T \Phi(x)$ is the unique matrix $C$
	satisfying
	$$ C \Phi(x) = x \Phi(x) - \cratio^{-1} p(x) e_1 $$
	where $\cratio = \coeff_n \coeff_{n-1}^{-1}.$
\end{definition}

In the following theorem, the second item is classical, \cite{barnett}. The first item also dates back to \cite{barnett}, although in a weaker form; it was stated in this form (without proof) in \cite{nn16}.  The third item may be new in this general form, although some special cases can be deducted from other published results, for example, if $\{\phi_j\}$ is the monomial basis it is a consequence of \cite{ddp} and for the Chebyshev basis it can be proved using the analysis of \cite{np17}. 
\begin{theorem} \label{thm:steffe}
	The following properties hold:
	\begin{enumerate}
		\item $C$ is a (strong) linearization of $p(x)$ 
		  (implying $\det(xI-C)=\frac{p(x)}{\coeff_n}$);
		\item $C$ can be written as $$C=H - \cratio^{-1} e_1 c^T$$ where $H$ is Hessenberg and only depends on the basis $\{\phi_j\}$;
		\item $\adj(x I - C) e_1 = \coeff_{n-1}^{-1} \Phi(x).$
	\end{enumerate}
\end{theorem}

\begin{proof} We prove the three points separately. 
	\begin{enumerate}
		\item It can be easily verified that $x I -C$ belongs to the vector space $\mathbb{L}_1$ for the basis $\{\phi_j\}$ \cite{4m, nnt17}. Since it is manifestly a nonsingular pencil, it is a strong linearization for $p$ by \cite[Theorem 2.1]{nnt17} (or \cite[Theorem 4.3]{4m} for the monomial basis). Since $\det(xI - C)$ is monic in the monomial basis, 
		and $p$ is monic in the basis $\{ \phi_j \}$, the equality
		$\det(xI - C) = p(x)/\coeff_n$ follows. 
		\item Let $H$ be the matrix that satisfies $H \Phi(x) = x \Phi(x) - e_1 \cratio^{-1} \phi_{n}(x).$ Since $x \phi_k(x)$ has degree $k+1$ for all $k=0,\dots,n-2$ it follows that $H$ is Hessenberg, by the degree-gradedness of $\{\phi_j\}$. Now,
		$$\cratio (H-C) \Phi = e_1 (p(x)-\phi_n(x)) = e_1 c^T \Phi(x).$$
		Since this relation holds over $\R(x)$, a fortiori it is still true as a relation over $\R$ after evaluating $\Phi(x)$ at any point. Thus, for any Vandermonde matrix $V$ we obtain
		$$\cratio (H-C) V = e_1 c^T V \Rightarrow \cratio (H-C) = e_1 c^T,$$
		as desired.
		\item By definition of $C$ we have
		$$\cratio (x I - C) \Phi = p(x) e_1.$$
		As $xI-C$ is regular, it is invertible over $\R(x)$. Hence we can premultiply by its inverse (using $\det(xI-C)=p(x)/\coeff_n$), to obtain
		$$ \coeff_{n-1}^{-1} \Phi(x) = \adj (xI - C) e_1.$$		
	\end{enumerate}
\end{proof}

\begin{remark}
	The matrix $H$ is a square submatrix of the multiplication matrix $M$ in \cite[Section 2]{nnt17} and it represents the multiplication-by-$x$ operator in the quotient space $\R[x]/\langle \cratio^{-1}\phi_n \rangle$.
\end{remark}


\begin{example}[Companion matrix]
	\label{ex:comp}
	Consider the monomial basis $\{\phi_j\}$, where $\phi_j(x)=x^j$. We have
	$\phi_j(x)=x\phi_{j-1}(x)$. As a consequence $x \Phi(x) = H \Phi(x) + x^n e_1, $ where $H$ is the
	downshift matrix, i.e. the matrix with only ones on the subdiagonal, and zeroes
	elsewhere.
\end{example}

\subsection{Comrade matrices} \label{sec:comrade-intro}

When $\{\phi_j\}$ are orthonormal on a closed interval $\subseteq \R$ and have positive leading coefficients, the three terms recurrence
$$ \phi_j(x) = (\alpha_j x + \beta_{j}) \phi_{j-1}(x) - \gamma_j \phi_{j-2}(x)$$
holds for all $j$ and for some $\beta_j \in \R$, $\alpha_j = \coeff_j \coeff_{j-1}^{-1} > 0,$ $\gamma_k = \coeff_j \coeff_{j-2} \coeff_{j-1}^{-2} > 0$ \cite[Theorem 3.2.1]{szego}. As a consequence, multiplication-by-$x$ is encoded by
$$ x \phi_{j-1}(x) = \frac{1}{\alpha_j} \phi_j(x) - \frac{\beta_j}{\alpha_j} \phi_{j-1}(x) + \frac{\gamma_j}{\alpha_j} \phi_{j-2}(x)$$
which immediately implies that in this case
$$ x \Phi(x) = H \Phi(x) + \alpha_n^{-1} \phi_n(x) e_1, $$
where $H$ is tridiagonal, and has positive subdiagonal/superdiagonal elements. In the case of an orthogonal basis, the confederate matrix is also known as the comrade matrix of $p$. Note moreover that, as displayed above, in this setting $\cratio = \alpha_n$, so that for $p(x)= \phi_n(x)  + c^T \Phi(x)$ it holds $$C = H - \alpha_n^{-1} e_1 c^T.$$
The matrix $H$ has the following form:
\[
H := \begin{bmatrix}
- \frac{\beta_n}{\alpha_n} & \frac{\gamma_n}{\alpha_n} \\
\frac{1}{\alpha_{n-1}} & - \frac{\beta_{n-1}}{\alpha_{n-1}} & \frac{\gamma_{n-1}}{\alpha_{n-1}} \\
& \ddots & \ddots & \ddots \\
&& \frac{1}{\alpha_{2}} & - \frac{\beta_{2}}{\alpha_{2}} & \frac{\gamma_{2}}{\alpha_{2}} \\
&&& \frac{1}{\alpha_{1}} & - \frac{\beta_{1}}{\alpha_{1}} \\
\end{bmatrix}.
\]
We remark that for polynomials represented in the Chebyshev basis the matrix $C$ is
called the colleague matrix. 

In addition, we note that since $\alpha_j$ and $\gamma_j$ are positive, it is possible
to perform a diagonal scaling to the matrix $H$ that makes it symmetric. 
Indeed, we can consider the matrix $D^{-1} H D$, where $D$ 
is any diagonal matrix with entries 
\begin{equation} \label{eq:dscaling}
d_k := \sqrt{\frac{\alpha_{1}}{\alpha_{n-k+1}} \prod_{i = 2}^{n-k+1} \gamma_i }. 
\end{equation}
Observe that, in particular, $d_n = 1$. This corresponds to choosing the
orthogonal basis $\tilde \phi_j(x) := d_{n-j}^{-1} \phi_j(x)$, having
formally set $d_0 := 1$. The scaled 
matrices are as follows:
\[
D^{-1} H D = \begin{bmatrix}
- \frac{\beta_n}{\alpha_n} & \sqrt{\frac{\gamma_n}{\alpha_n \alpha_{n-1}}} \\
\sqrt{\frac{\gamma_n}{\alpha_n \alpha_{n-1}}} & - \frac{\beta_{n-1}}{\alpha_{n-1}} & \sqrt{\frac{\gamma_{n-1}}{\alpha_{n-1} \alpha_{n-2}}} \\
& \ddots & \ddots & \ddots \\
&& \sqrt{\frac{\gamma_3}{\alpha_3 \alpha_{2}}} & - \frac{\beta_{2}}{\alpha_{2}} & \sqrt{\frac{\gamma_2}{\alpha_2 \alpha_{1}}} \\
&&& \sqrt{\frac{\gamma_2}{\alpha_2 \alpha_{1}}} & - \frac{\beta_{1}}{\alpha_{1}} \\
\end{bmatrix}, 
\]

\[ D^{-1} C D = D^{-1} H D - \tilde{\cratio}^{-1} e_1 \tilde{c}^T \]
where $\tilde{c}$ is the vector of the coefficients of $p$ expressed in the 
scaled basis $\{ \phi_0, \tilde{\phi}_1, \dots, \tilde{\phi}_{n-1}, \phi_n  \}$ and $\tilde{\cratio} = \sqrt{\alpha_1 \alpha_n \gamma_2 \gamma_3 \cdots \gamma_n}$; the coefficient
$\tilde{\cratio}$ is the analogue of $\cratio$ for the rescaled basis $\{ \phi_0, \tilde{\phi}_1, \dots, \tilde{\phi}_{n-1}, \phi_n  \}$. From now on we work in this symmetrized setting, and we only consider
$D^{-1} C D$. From the viewpoint of developing structured algorithms, this is 
particularly relevant. If $A = H + uv^T$, with $H$ real symmetric or Hermitian and 
$uv^T$ of rank $1$, then all the matrices obtained 
through the iteration of a QR method, that can be written as $A_k := Q_k A Q_k^H$, with
$Q_k$ orthogonal or unitary, have the same property. 

This observation is key in the development of fast algorithms; in the monomial case, 
the companion matrix can be similarly decomposed as the sum of a unitary and 
a rank $1$ part; this property is also preserved by QR iterations. 

Fast algorithms for these classes of matrices often work on the structured (either
Hermitian or unitary) and rank one part separately. Therefore, it is reasonable to 
assume that these parts might be contaminated, throughout the iterations, by 
backward errors of different magnitudes. Classical backward error analysis does not
take this property into account, so we present a more general backward error 
formulation in the next section. 

\section{Mixed backward error analysis}
\label{sec:mberr}

We are now ready to study the behavior of the linearized polynomial $p(x)$
under perturbations to the pencil $xI - C$.  More precisely, we consider
$xI - (C + \delta C)$ where $C + \delta C$ has a mixed backward error of the following form:
\begin{equation}
\label{eq:pC}
C + \delta C = H + \delta H + (e_1 + \delta e_1) (c + \delta c)^T. 
\end{equation}

We note that, for any $\delta C$, it is always possible to find a decomposition
as the one in \eqref{eq:pC}. Indeed, for any choice of $\delta e_1$ and $\delta c$
it suffices to choose $\delta H := \delta C - \delta e_1 c^T - e_1 \delta c^T -
\delta e_1 \delta c^T$. Our analysis will show that these terms provide
different contributions to the backward error on the polynomial, with different
amplification factors. In particular, it will show that errors on $\delta H$ 
may be amplified much more when projecting the error back on the polynomial, whereas 
the backward errors $\delta e_1$ and $\delta c$ are relatively less harmful. 

By Theorem~\ref{thm:steffe} the perturbed matrix $C$ linearizes 
the polynomial 
\begin{equation} \label{eq:ppdp}
p(x) + \delta p(x) := \coeff_n \det(xI - C - \delta C). 
\end{equation}
Our aim
is now to examine the size of $\delta p(x)$, under the assumption that for the various
actors in \eqref{eq:pC} a bound is known.  
In particular, we 
assume to know appropriate positive $\epsilon_H, \epsilon_1, \epsilon_c$ such that
\begin{equation} \label{eq:epsilon}
\norm{\delta H}_2   \leq \epsilon_H < 1, \qquad 
\norm{\delta c}_2   \leq \epsilon_c, \qquad 
\norm{\delta e_1}_2 \leq \epsilon_1. 
\end{equation}
Observe that, as discussed above, infinitely many choices exist, given $\delta C$, for $\delta H, \delta c, \delta e_1$: thus, we may without any loss of generality assume to have picked one that is optimal (in the sense of making our results as sharp as possible) for the values of $\epsilon_H$, $\epsilon_c$, $\epsilon_1$. Of course, it may also happen that in practice an algorithm suggests one particular choice for which bounds on $\epsilon_H$, $\epsilon_c$, $\epsilon_1$ are ``naturally" obtained (see e.g. \cite{aurentz2017fast}).

In this section we will discuss the general
setting, which holds for all confederate matrices. 
In Sections~\ref{sec:monomial}
and \ref{sec:comrade} we will specialize to companion and comrade matrices, that is either the monomial basis, or a
basis of polynomials orthogonal on a real interval. 

We will make the following assumptions in our analysis:
\begin{description}
	\item[Assumption 1] The matrix $H$ is \emph{normal}; for instance, we will consider the case where 
	  $H$ is unitary (the monomial basis) and symmetric (the Chebyshev case and 
	  in general orthogonal polynomials on the real line). 
	\item[Assumption 2] 
	  The backward errors on $H, e_1, c$ may be of very \emph{different magnitude}. In
	  particular, to obtain strong backward stability on the polynomial, we will need  $\epsilon_H$ and $\epsilon_1$
	  to be bounded \emph{independently} of the norm of the vector 
	  of polynomial coefficients; on the contrary, $\epsilon_C$ may
	  depend linearly on this value. 
\end{description}
We stress that the second assumption, in particular the bound on $\epsilon_H$, is the 
one which is hard to obtain in practice when designing a structured algorithm. It does
not hold for the unstructured QR or QZ, and obtaining it was key in 
developing a backward stable algorithm for the monomial case in \cite{aurentz2017fast,aurentz2015fast}; we hope that our analysis will help to devise
similar algorithms in other bases, in particular in the Chebyshev one. 

From now on, we use the notation $\doteq$ to indicate an equality that holds up
to second order in the terms $\epsilon_H, \epsilon_1, \epsilon_c$. 
Based on the expressions above, we can rewrite the perturbed polynomial.
\begin{theorem} 
	\label{thm:pert}
	With the notation of \eqref{eq:pC}, \eqref{eq:ppdp}, and \eqref{eq:epsilon}, the following first order
	expansion in $\epsilon_H, \epsilon_c, \epsilon_1$ holds:
	\begin{eqnarray*} 
		(p+\delta p)(x) & \doteq &  p(x) + \coeff_n \left[\det(xI - H - \delta H) - \det(xI - H)\right] \\
		&       &+ \cratio \delta c^T \Phi(x) + \coeff_n c^T \adj(xI - H) \delta e_1 + \\
		&       &+ \coeff_n c^T \left[ \adj(xI - H - \delta H) - \adj(xI - H) \right] e_1.
	\end{eqnarray*}
\end{theorem}

\begin{proof}
	By \eqref{eq:ppdp} we have $(p+\delta p)(x) = 
	\coeff_n \det (xI - (C + \delta C))$, which 
	by Theorem~\ref{thm:steffe}
	is equal to 
	\[
	(p+\delta p)(x) = \coeff_n \det(xI - H - \delta H) + \coeff_n (c+\delta c)^T 
	\adj(xI - H - \delta H) (e_1 + \delta e_1). 
	\]
	To obtain the statement, we first add $p(x)$ and subtract
	its expansion obtained by Theorem~\ref{thm:steffe}. Next, we discard
	higher order terms, and use
	the equalities $\coeff_{n-1} \adj(xI-H) e_1 = \Phi(x)$
	and $\coeff_{n-1} \cratio = \coeff_n$.
\end{proof}

Theorem~\ref{thm:pert} reveals that, in order to provide bounds for the
perturbation $\delta p(x)$, it is essential to do a 
perturbation analysis related to determinants and adjugates.
To this aim, 
we provide a few results that will be useful in later proofs.

\begin{lemma}[Jacobi's formula] \label{lem:detexpansion}
	Let $X$ be any square matrix, and 
	$\delta X$ 
	a small perturbation. Then, 
	\[
	\det(X + \delta X) = \det(X)  + \tr(\adj(X) \cdot \delta X) + O(\norm{\delta X}^2). 
	\]
\end{lemma}

A similar result can be given for the adjugate as well, and characterizes
the effect of small perturbations. 

\begin{lemma} \label{lem:adjexpansion}
	Let $\delta X$ be a small perturbation ($\|\delta X\| < 1$). Then, 
	\begin{equation}
	\label{eq:bound:adj}
	\adj(I + \delta X) = (I - \delta X) \cdot (1 + \tr(\delta X)) + O(\norm{\delta X}^2). 
	\end{equation}
\end{lemma}

\begin{proof}
	Since $\norm{\delta X} < 1$, $I + \delta X$ is
	invertible and therefore we can write
	\[ \adj(I + \delta X) = (I + \delta X)^{-1} \cdot
	\det(I + \delta X) .
	\] We shall make a first order approximation of both
	terms involved in the above equality. Concerning the first one, we
	have that $(I + \delta X)^{-1} \doteq I - \delta X$. To bound the
	change in the determinant we  use 
	Lemma~\ref{lem:detexpansion} and obtain
	\[ \det(I + \delta X) = 1 + \tr(\delta X) +
	O(\norm{\delta X}^2),
	\] which provides the sought first-order expansion \eqref{eq:bound:adj}.
\end{proof}

\begin{lemma} \label{lem:traceAB}
	Let $A, B$ be two $n \times n$ matrices, and assume that $A$
	is normal with eigenvalues $\lambda_1, \ldots, \lambda_n$. Then,
	\[
	|\tr(AB)| \leq \norm{B}_2 \cdot \sum_{j = 1}^n |\lambda_j|
	\]
\end{lemma}

\begin{proof}
	Let $A = QDQ^{H}$ be an eigendecomposition of $A$, with $Q$ unitary. Then,
	if we denote by $q_j$ the columns of $Q$ we can write:
	\begin{align*}
	|\tr(AB)| &= |\tr(Q^H A Q Q^H B Q)| = |\tr(D Q^H B Q)| \leq
	|\sum_{j = 1}^n \lambda_{j} q_j^H B q_j|. 
	\end{align*}
	Since $|q_j^H B q_j| \leq \norm{B}_2$, the result follows.
\end{proof}

With these tools at hand, we are now able to bound the 
point-wise perturbation
$\delta p(\xi)$, i.e., the evaluation of the perturbation at 
any point $\xi \in \mathbb C$. Later on, when going to comrade and companion
matrices we will need to specify these points $\xi$ to retrieve tight bounds
on the polynomials' coefficients. 

\begin{lemma} \label{lem:pointwisebound}
	Let $(p + \delta p)(x) = \det(xI - C - \delta C)$ be the perturbed
	polynomial.
	Then, for every $\xi \in \mathbb C$ we get the following first order
	bound:
	\[
	|\delta p(\xi)| \leq \Gamma_1(\xi) \epsilon_1 + \Gamma_c(\xi) \epsilon_c + \Gamma_H(\xi) \epsilon_H + \mathcal O(\epsilon_H^2 + \epsilon_1^2 + \epsilon_c^2),
	\]
	where 
	\begin{align*}
	\Gamma_1(\xi) &:= M(\xi) |\phi_n(\xi)| \cdot \norm{c}_2, &
	\Gamma_c(\xi) &:= \cratio \norm{\Phi(\xi)}_2, \\
	\Gamma_H(\xi) &:= S(\xi)  |\phi_n(\xi)| + \Gamma_c(\xi)
	\left[M(\xi)+S(\xi)\right]  \norm{c}_2,
	\end{align*}
	having defined
	\[
	S(\xi) := \sum_{j = 1}^n \frac{1}{|\xi - r_j|}, \qquad 
	M(\xi) := \max_{j = 1, \ldots, n} \frac{1}{|\xi - r_j|}. 
	\]
	In the expressions for $S(\xi)$ and $M(\xi)$ above, $r_j$ denote the roots of the  polynomial $\phi_n$ of degree $n$. 
\end{lemma}

\begin{proof}
	Let us first note that, since 
	$\phi_n(x) = \coeff_n \det(xI - H)$ and
	since $\xi I - H$ is normal we have 
	\[
	\norm{(\xi I - H)^{-1}}_2 = \max_{j = 1, \ldots, n} \frac{1}{|\xi - r_j|}. 
	\]
	By Theorem~\ref{thm:pert} we have the following first order approximation
	for $\delta p(\xi)$: 
	\begin{align}
	\label{eq:exp1}  \delta p(\xi) &\doteq \coeff_n \left[ \det(\xi I - H - \delta H) - \det(\xi I - H) \right] \\
	\label{eq:exp2}           &+ \cratio \delta c^T \Phi(\xi) + \coeff_n c^T \adj(\xi I - H) \delta e_1 \\
	\label{eq:exp3}           &+ \coeff_n c^T \left[ \adj(\xi I - H - \delta H) - \adj(\xi I - H) \right] e_1. 
	\end{align}
	We bound all the terms separately.
	\begin{itemize}
		\item We consider \eqref{eq:exp1} first. By 
		Lemma~\ref{lem:detexpansion} we can write
		\[
		\det(\xi I - H - \delta H) - \det(\xi I - H) \doteq \tr(\adj(\xi I - H)
		\delta H)
		\]
		Since $\xi I - H$ is a normal matrix, we can use Lemma~\ref{lem:traceAB} to
		obtain
		\[
		\coeff_n | \det(\xi I - H - \delta H) - \det(\xi I - H) | \doteq \coeff_n |\tr(\adj(\xi I -
		H) \delta H)| \leq \epsilon_H \cdot \sum_{j = 1}^n
		\frac{\phi_n(\xi)}{|\xi - r_j|}.
		\]
		\item To bound the second term in (\ref{eq:exp2}), we 
		use 
		\[
		\norm{\adj (\xi I - H)}_2 = |\det(\xi I - H)| \cdot 
		\norm{(\xi I - H)^{-1}}_2 = 
		\frac{|\phi_n(\xi)|}{\coeff_n} \max_{j} \frac{1}{|\xi - r_j|}. 
		\]
		Bounding the first term 
		just requires taking norms
		of all the factors involved.
		\item To bound (\ref{eq:exp3}), assuming $\|(\xi I-H)^{-1}
		\delta H\| \leq 1$ and using Lemma~\ref{lem:adjexpansion}, we write 
		\begin{align*}
		\adj(\xi I - H - \delta H) &= \adj\left[(\xi I - H) \cdot (I -  (\xi I - H)^{-1} \delta H)\right] \\
		&= \adj(I - (\xi I - H)^{-1}\delta H) \adj(\xi I - H)  \\
		&\doteq (I +  (\xi I - H)^{-1}\delta H)
		\cdot (1 - \tr( (\xi I - H)^{-1}\delta H)) \adj(\xi I - H) . 
		\end{align*}
		Then, using the fact that $\coeff_{n-1} \adj(\xi I - H) e_1 = \Phi(\xi)$, we have
		\begin{eqnarray*}
			\coeff_n \lefteqn{c^T \left[ \adj(\xi I - H - \delta H) - \adj(\xi I - H) \right] e_1} \\
			& \doteq &
			\cratio c^T \left[  (\xi I - H)^{-1} \delta H - 
			\tr( (\xi I - H)^{-1} \delta H) I \right] \Phi(\xi).
		\end{eqnarray*}
	\end{itemize}
	Taking norms and combining all the results yields the desired bound.
\end{proof}

The results we have proved are valid for any class of polynomials under the
assumption that $H$ is normal\footnote{We note that these perturbations might make
	$H + \delta H$ non-normal, but this is not an issue for our analysis, as we 
	always deal with expansions centered at $H$.}. We
will use this idea to generalize the point-wise bound to a bound on
the coefficients in the case of the monomial basis
and of orthogonal polynomials on a real
interval. These are the subjects of the next sections.




\section{Companion matrix} \label{sec:monomial}
In this section we reconsider the error analysis of Aurentz et al. \cite{aurentz2017fast}, in view
of this new theory. The derivation of \cite{aurentz2017fast} is based 
on running the Faddev-Leverrier algorithm to compute the coefficients
of the adjugates, and uses it to provide bounds on its norm. This approach
is not easily generalizable, despite the existence of a Faddev-Leverrier scheme
for nonmonomial bases.  Our new point of view yields a simple and clean
derivation of the results therein, based instead on an interpolation argument. 

To analyze the backward error of an algorithm running on the companion matrix, we have to
rewrite the companion matrix slightly. Example~\ref{ex:comp} revealed that the Hessenberg matrix $H$
is the downshift matrix, and the eigenvalues can be retrieved from $C=H- e_1 c^T$, i.e.
the downshift matrix plus a rank one part. Structure exploiting algorithms, however,  rely on the unitary-plus-low
rank structure, and rewriting $C=\tilde{H} - e_1 \tilde{c}^T$, with $\tilde{H}=H- e_1 e_n^T$ and
$\tilde{c}=c+e_n^T$ is clearly of unitary-plus-low rank form.

This has some impact on the backward error, since we are now working with the basis
$1,x,\ldots,x^{n-1},x^n+1$, instead of the classical monomial basis. Moreover, also the
trailing coefficient of our polynomial $p$ has changed. For simplicity we will therefore, from now on, assume to
be working in the basis $1,x,\ldots,x^{n-1},x^n+1$.

Eventually we will use the fast Fourier transform to retrieve the coefficients of $\delta
p$. To do so, we need to bound $\delta p$ evaluated in the $n$-th roots of
unity $\xi_j$, for $j=0,\ldots,n-1$. Lemma~\ref{lem:pointwisebound} provides
\begin{equation*}
|\delta p(\xi_j)| \leq \Gamma_1(\xi_j) \epsilon_1 + \Gamma_c(\xi_j) \epsilon_c + \Gamma_H(\xi_j) \epsilon_H.
\end{equation*}
Using the formulas of Lemma~\ref{lem:pointwisebound} yields
$\Gamma_1(\xi_j)=2M$, and $\Gamma_H
\leq 2S+\sqrt{n+3}(M+S) \|c\|_2$. Recalling that $\phi_n(x)=x^n+1$ we have 
that $r_j = e^{\frac{(2j+1)\pi}{n}}$. Exploiting that $|\phi_i(\xi_j)| = 1$ for $i < n$ 
and $\cratio = 1$, we obtain $\Gamma_c(\xi_j) = \lVert \Phi(\xi_j) \rVert_2 \leq \sqrt{n}$. The quantity $M$ is the inverse 
between the distance of the $n$-th roots of the unity and $r_j$, which can be bound
by $M \leq \frac{n}{2}$.
To bound $S=\sum_{k=1}^n \frac{1}{|\xi_j - r_k|}$, we use 
\begin{equation*}
S=\sum_{k=1}^n \frac{1}{|\xi_j - r_k|}
= \sum_{k=1}^n \frac{1}{|1 - e^{\frac{2\pi}{2n}(2j+1)}|}
= \sum_{k=1}^n \frac{1}{\left|2 \sin\left(\frac{\pi}{2n}(2j+1) \right) \right|}
= 2\sum_{k=1}^{n/2} \frac{1}{2 \sin\left(\frac{\pi}{2n}(2j+1) \right) }
\end{equation*}
and the fact that $\sin x > \frac{2}{\pi} x$ for $x\in[0,\frac{\pi}{2}]$. As a result we
obtain 
\begin{equation*}
S \leq \frac{n}{2} \log\left(\frac{n}{2}+\frac{1}{2}\right).
\end{equation*}
Combining all of this leads to
\begin{equation}
\label{eq:bnd}
\begin{aligned}
|\delta p(\xi_j)| &\leq 
n \norm{c}_2 \epsilon_1 +\sqrt{n}\, \epsilon_c + n \log\left(\frac{n}{2}\right) \epsilon_H
\\
&+ \frac{n\sqrt{n}}{2}\left(1 + \log\left(\frac{n}{2}+\frac{1}{2}\right)\right) \|c\|_2 \epsilon_H
+ \mathcal O(\epsilon_H^2 + \epsilon_1^2 + \epsilon_c^2).
\end{aligned}
\end{equation}

As a result, we get as Euclidean norm on the coefficients of $\delta p$, denoted as
$\|\delta p\|_2$, 
\begin{equation*}
\|\delta p\|_2 = 
\left\| \frac{1}{\sqrt{n}} F^* q \right\|_2 
\leq \| q \|_\infty+ \mathcal O(\epsilon_H^2 + \epsilon_1^2 + \epsilon_c^2), 
\end{equation*}
where $q=[\delta p(\xi_0),\ldots, \delta p(\xi_{n-1})]^T$, and $F$ is the
matrix of the discrete
Fourier transform. The last factor can be bounded
by  \eqref{eq:bnd}.

Reconsidering the algorithm of Aurentz et al.\ \cite{aurentz2017fast}, we have that
$\epsilon_1=0$, $\epsilon_H=\epsilon_m$,  and $\epsilon_c= \|c\|_2 \epsilon_m$,
where $\epsilon_m$ is the machine precision. Clearly we end up with the same bound as
proposed by Aurentz et al., namely a linear dependency on $\|c\|_2$.

Before moving to orthogonal basis on real intervals, and in particular Chebyshev
and Jacobi polynomials, we emphasize that the main ingredients playing a role in 
the bound are related to the eigenvalues of the structured matrix $H$, namely their
separation, as measured by the constants $M,S$ of Lemma~\ref{lem:pointwisebound}, 
and their good properties as interpolation points for the chosen basis. These
two quantities will play an important role in the analysis of the following sections as well.

\section{Orthogonal polynomials on a real interval}
\label{sec:comrade}
\label{sec:orthoreal}

In this section, we consider a class of degree-graded polynomials $\phi_i(x)$,
for $j \geq 0$, that are orthogonal on $[-1, 1]$ with respect to a positive
measure $w(x)$.

Our aim is to leverage Lemma~\ref{lem:pointwisebound} to provide
a bound on the coefficients of the perturbed polynomial $\delta p(x)$.
To this aim, we provide the following result, which holds for any
polynomial family orthogonal on $[-1, 1]$; since this bound
is not very explicit, we will then specialize it to a few
particular families of polynomials for which we can be more
precise, namely Chebyshev and later on all Jacobi polynomials. 

\begin{theorem} \label{thm:structuredbe}
	In the notation of \eqref{eq:pC} and \eqref{eq:epsilon},
	let $\{ \phi_i \}$ be a basis of orthogonal
	polynomials on $[-1, 1]$, such that $H$ is real and 
	symmetric. 
	Let $\{\rho_j \}^n_{j=0}$ be distinct points in $[-1, 1]$, and $\{r_j\}_{j=1}^n$ the roots of $\phi_n(x)$. 
	Let $\{ \ell_j(x) \}_{j = 0}^n$ 
	be the Lagrange polynomials defined 
	on the nodes $\rho_0, \ldots, \rho_n$, and consider the matrix $L$
	such that $L_{ij}$ contains the $i$-th coefficient of $\ell_j(x)$ with
	respect to the basis $\{ \phi_i \}$. 
	Then, the norm of the 
	vector of coefficients of $\delta p(x)$ can be bounded by 
	\[
	\norm{\delta p}_\infty \leq \norm{\hat L}_\infty \cdot 
	\left( \max_{j=0,\ldots,n} 
	\Gamma_1(\rho_j) \epsilon_1 + \Gamma_c(\rho_j) \epsilon_c + \Gamma_H(\rho_j) \epsilon_H \right) + 
	\mathcal O(\epsilon_H^2 + \epsilon_1^2 + \epsilon_c^2), 
	\]
	where $\hat L$ is the matrix with the first $n$ rows 
	of $L$, and $\Gamma_1, \Gamma_c, \Gamma_H$ are defined
	as in Lemma~\ref{lem:pointwisebound}. 
\end{theorem}

\begin{proof}
	We note that $\delta p(x) = \sum_{j = 0}^{n-1} \delta p_j \phi_j(x)$ is a
	degree $n - 1$ polynomial. Its coefficients can be recovered by interpolation on the points $\{ \rho_0,  \ldots, \rho_{n} \}$.
	Notice that these are $n + 1$ points, one more than actually required. Let $V_{n}$ be the 
	$(n + 1) \times (n+1)$ generalized Vandermonde matrix interpolating
	on these nodes in the prescribed basis. Hence, we have 
	\[
	\begin{bmatrix}
	\delta p_0 \\
	\vdots \\
	\delta p_{n - 1} \\
	0 
	\end{bmatrix} = V_n^{-1} \begin{bmatrix}
	\delta p(\rho_0) \\
	\delta p(\rho_1) \\
	\vdots \\
	\delta p(\rho_n)
	\end{bmatrix}. 
	\]
	Note that $L = V_{n}^{-1}$. Indeed, the entries of the inverse of a 
	Vandermonde matrix are the coefficients of the Lagrange
	polynomials with nodes $\rho_0, \ldots, \rho_n$. Therefore, we 
	have, for $0 \leq i \leq n - 1$, 
	\[
	|\delta p_i| \leq
	\sum_{j = 1}^{n+1} |L_{i+1,j}| \cdot |\delta p(\rho_j)| \leq 
	\norm{\hat L}_{\infty} \cdot \max_{0 \leq j \leq n} |\delta p(\rho_j)|, 
	\]
	where with $\hat L$ we denote the first $n$ rows of $L$. The
	statement then follows by applying Lemma~\ref{lem:pointwisebound}. 
\end{proof}

\subsection{Chebyshev polynomials} \label{sec:chebyshev}

Chebyshev polynomials of the first kind
play a special role among orthogonal polynomials
on $[-1, 1]$, in particular thanks to their nice approximation properties. 
For instance, they are the basis of the \texttt{chebfun} MATLAB toolbox \cite{driscoll2014chebfun}, 
that aims at making computing with functions as accessible as computing
with matrices and vectors. 

Their orthogonality measure is defined by
the weight function $w(x) = (1 - x^2)^{-\frac{1}{2}}$, 
and they can be obtained through the recursive relations
\[
T_{k+1}(x) = 2x T_k(x) - T_{k-1}(x), \qquad 
T_0(x) = 1, \qquad 
T_1(x) = x. 
\]
We denote by $U_k(z)$ the Chebyshev polynomials of the second kind, 
which can be
obtained replacing the degree $1$ polynomial with $2x$, and keeping 
the rest of the recursion unchanged. The latter are orthogonal
with respect to the weight $\sqrt{1 - x^2}$. Moreover, 
$T_n'(x) = n U_{n-1}(x)$, and therefore the extrema of $T_n(x)$ 
are the roots of $U_{n-1}(x)$. 

Our aim in this section is to apply Theorem~\ref{thm:structuredbe} 
to Chebyshev polynomials of the first kind, making all the involved
constants explicit, or functions of the degree. To this aim, we need
to choose the interpolation nodes, and in this case we select
$\rho_j = \cos(j \pi / n)$, for $j = 0, \ldots, n$, which are the 
roots of $U_{n-1}(x)$ (with, additionally, the points $\pm 1$) 
and therefore the extrema of $T_n(x)$ on $[-1, 1]$. 

\begin{lemma} \label{lem:norml-chebyshev}
	Let $\hat L$ be the matrix defined as in Theorem~\ref{thm:structuredbe} choosing as 
	$\{ \phi_j \}$ the Chebyshev polynomials of the first kind, 
	and as nodes $\rho_j = \cos(j \pi / n)$, for
	$j = 0, \ldots, n$. Then, $\norm{\hat L}_\infty \leq 2$. 
\end{lemma}

\begin{proof}
	We prove the result by showing that, for 
	$1 \leq i \leq n$ 
	we have $|\hat L_{ij}| \leq \frac{2}{n}$
	if
	$2 \leq j \leq n$, and $|\hat L_{ij}| \leq \frac{1}{n}$ 
	for $j \in \{1, n+1\}$. It 
	immediately follows that the row sums of $|\hat L|$ are bounded
	by $2$, 				and thus the claim holds. 
	
	For any $i,j$, since 
	$\hat L_{ij}$ is the Chebyshev coefficients corresponding to $T_{i-1}$
	of $\ell_{j-1}(x)$, we can recover it by writing
	\[
	\norm{T_{i-1}(x)}^2 \cdot \hat L_{ij} = \int_{-1}^1 \frac{\ell_{j-1}(x) T_{i-1}(x)}{\sqrt{1 - x^2}}\ dx, \qquad 
	i,j = 1, \ldots, n + 1. 
	\]
	Here $\norm{T_{i-1}(x)}$ denotes the norm induced by the scalar
	product considered above. 
	We note that, if $2 \leq j \leq n$, then $\ell_{j-1}(x)$ is divisible
	by $(1 - x)^2$, since it vanishes at $\pm 1$. Therefore, for $1 \leq j \leq n - 1$, we can
	define the degree $n - 2$ polynomial $q_j(x) := \ell_j(x) / (1 - x^2)$
	and rewrite the formula as follows: 
	\[
	\hat L_{ij} = \frac{1}{\norm{T_{i-1}(x)}^2} \cdot 
	\int_{-1}^1 q_{j-1}(x) T_{i-1}(x) \sqrt{1 - x^2}\ dx, \qquad 
	2 \leq j \leq n. 
	\]
	Since $\deg(q_{j-1}(x) T_{i-1}(x)) = n + i - 3 \leq  2n - 3$, 
	because we are assuming $i \leq n$, we can
	integrate the above exactly using a Chebyshev-Gauss quadrature formula
	with Chebyshev polynomials of the second kind of degree $n - 1$, 
	which yields 
	\[
	\norm{T_{i-1}(x)}^2 \cdot \hat L_{ij} = 
	\sum_{s = 1}^{n - 1} \frac{w_s}{1 - x_s^2} \ell_{j-1}(x_s) T_{i-1}(x_s) = 
	\frac{w_{j-1}}{1 - x_{j-1}^2} T_{i-1}(x_{j-1}). 
	\]
	For Chebyshev-Gauss quadrature of the second kind, the $w_s$ 
	are known explicitly and are $w_s = \frac{\pi}{n} (1 - x_s^2)$; this, 
	combined with $\norm{T_{i-1}(x)}^2 \geq \frac{\pi}{2}$ 
	and $|T_{i-1}(x_{j-1})| \leq 1$ 
	yields $|\hat L_{ij}| \leq \frac{2}{n}$. 
	
	It remains to consider the case $j \in \{ 1, n + 1 \}$. 
	Without loss of generality
	we can consider $j=1$, which is associated with 
	$\ell_0(x)$. Since $\ell_0(x)$ has as roots the zeros of $U_{n-1}(x)$ and $-1$, we can write
	it as $\ell_0(x) = \gamma (1 + x) U_{n-1}(x)$ up to a scaling factor $\gamma$. The latter
	can be determined imposing $\ell_0(\rho_0) = \ell_0(1) = 1$ which yields
	$\gamma = (2n)^{-1}$ since $U_{n-1}(\pm 1) = n$. 
	Similarly, we can show that $\ell_{n}(x) = (2n)^{-1}(1-x)U_{n-1}(x)$. In addition, we may write
	\[
	(1 + x) U_{n-1}(x) = \sum_{j = 0}^{n} f_j T_j(x),  \qquad 
	(1 - x) U_{n-1}(x) = \sum_{j = 0}^{n} (-1)^{n-j+1} f_j T_j(x), 
	\]
	where $f_j = 2$ if $1 \leq j \leq n-1$, and $1$ if $j \in \{1,n\}$.
	These equalities can be easily verified using relation (22.5.8) from \cite[page 778]{abramowitz1965handbook}. Hence, we can conclude that 
	$|\hat L_{i1}| = |\hat L_{i,n+1}| \leq \frac{1}{n}$, and 
	therefore $\norm{\hat L}_\infty \leq (n-1) \cdot \frac{2}{n} + \frac 1n + \frac 1n = 2$. 
\end{proof}

To apply Theorem~\ref{thm:structuredbe} we need to obtain bounds for the
constants $\Gamma_1, \Gamma_c$ and $\Gamma_H$, 
which in turn requires to bounds the quantities $M$ and $S$ as defined in 
Lemma~\ref{lem:pointwisebound}. 

\begin{lemma} \label{lem:boundcheb}
	For Chebyshev polynomials, with the notation of Lemma~\ref{lem:pointwisebound}, and $\xi = \rho_j$ as 
	defined in Theorem~\ref{thm:pert}, we have 
	\[
	M \leq 3n^2, \qquad 
	S \leq 5n^2. 
	\]
\end{lemma}

The above result is somewhat tedious to prove, so we delay the proof to Section~\ref{sec:prooflem}; it  
allows to state the following corollary for the case of Chebyshev polynomials. 
Recall that, given a monic polynomial $p(x) = \sum_{j = 0}^n p_j T_j(x)$, the  
(scaled) colleague matrix is given by: 
\begin{equation} \label{eq:cheblin}
C = H - \frac{1}{2} e_1 c^T = \begin{bmatrix}
0 & \frac{1}{2} \\
\frac{1}{2} & 0 & \ddots \\
& \ddots & \ddots & \frac{1}{2} \\
& & \frac 12 & 0 & \frac{\sqrt{2}}{2} \\
& & & \frac{\sqrt{2}}{2} & 0 \\
\end{bmatrix} - \frac{1}{2} e_1 \begin{bmatrix}
p_{n-1} & \ldots & p_1 & \sqrt{2} p_0
\end{bmatrix}, 
\end{equation}
as described in Section~\ref{sec:comrade-intro}, since for Chebyshev polynomials
of the first kind we have $\alpha_n = 2, \beta_n = 0, \gamma_n - 1$, with
the only exception of $\alpha_1 = 1$. 

\begin{corollary} \label{cor:chebfinal}
	Let $C = H - \cratio^{-1} e_1 c^T$ the scaled linearization for a polynomial
	$p(x)$ expressed in the Chebyshev basis given by \eqref{eq:cheblin}. 
	Consider perturbations 
	$\norm{\delta H}_2 \leq \epsilon_H$, $\norm{\delta e_1} \leq \epsilon_1$, 
	and $\norm{\delta c} \leq \epsilon_c$. Then, the matrix $C + \delta C := H + \delta H - \cratio^{-1} (e_1 + \delta e_1) (c + \delta c)^T$ linearizes 
	the polynomial \[
	p(x) + \delta p(x) := \sum_{j = 0}^n (p_j + \delta p_j) T_j(x), 
	\]
	where $|\delta p_j| \leq 
	\left(
	6 \norm{c}_2 \epsilon_1 + 2 \sqrt{n} \epsilon_c + (5 + 16\sqrt{n}\norm{c}_2) \epsilon_H
	\right) n^2
  + \mathcal O(\epsilon_H^2 + \epsilon_1^2 + \epsilon_c^2)
	$.
\end{corollary}

\begin{proof}
	This result follows combining Lemma~\ref{lem:pointwisebound} with
	Theorem~\ref{thm:pert} and Lemma~\ref{lem:boundcheb}. More precisely, the
	bound is obtained on the coefficients of the polynomial 
	\[
	p + \delta p(x) = \sum_{j = 0}^n (q + \delta q_j) \tilde T_j(x), 
	\]
	where $T_0(x) = \sqrt{2}^{-1} T_0(x)$ and 
	$\tilde T_j(x) = T_j(x)$ otherwise. Therefore, we have $\delta p_j = \delta q_j$ and $\delta p_0 = \sqrt{2} \sqrt{2}^{-1} \delta q_0$, so in particular
	$|\delta p_0| \leq |\delta q_0|$. 
\end{proof}

The previous result tells us that
a structured QR algorithm working on the Hermitian and rank one part separately, 
and ensuring a low relative backward error on these two components, 
would give a backward stable rootfinding algorithm. Indeed, 
in that case we would have 
\[
\epsilon_H \lesssim \norm{H}_2 \epsilon_m, \qquad 
\epsilon_1 \lesssim \epsilon_m, \qquad 
\epsilon_c \lesssim \norm{c}_2 \epsilon_m, 
\]
where $\lesssim$ is used to denote the first order inequality up to a constant
and a low-degree polynomial in the degree. Combining this fact with the result 
Corollary~\ref{cor:chebfinal} would guarantee that the backward error on
the polynomial is bounded by $\norm{\delta p} \lesssim (1 + \norm{p}) \epsilon_m$. 

Before providing the details of the proof, we check experimentally 
the results of Corollary~\ref{cor:chebfinal} by generating 
polynomials expressed in the Chebsyhev basis and measuring the impact of 
perturbing $H$, $e_1$, and $c$ in the (scaled) colleague
linearization. More precisely, we have generated polynomials 
$p(x) = \sum_{j = 0}^n p_j T_j(x)$ with $n = 5$; the $p_j$ have 
been specifically designed to be relatively unbalanced, a configuration 
that often triggers worst case behaviors in QR-based rootfinders. More 
specifically, we have set: 
\[
  p_j = \begin{cases}
    \gamma_j \cdot 3^{5.5 \eta_j} & j < n \\
    1 & j = n
  \end{cases}, \qquad 
  \gamma_j, \eta_j \sim N(0, 1), 
\]
and we have perturbed the terms $H,u,v$ with perturbations of 
relative norm $10^{-6}$. Our motivation for this choice of the coefficients distribution and the 
perturbation norm is that we wanted to explore difficult examples and check if yet we could retrieve meaningful results in floating  
point; the backward error has been computed in higher precision 
starting from the eigenvalues relying on the MPFR library \cite{mpfr}.  The results, showing the actual backward errors and the bounds, 
are reported in Figure~\ref{fig:cheb-examples}; in each experiment
we have perturbed only one of the input data $H, u, v$. 

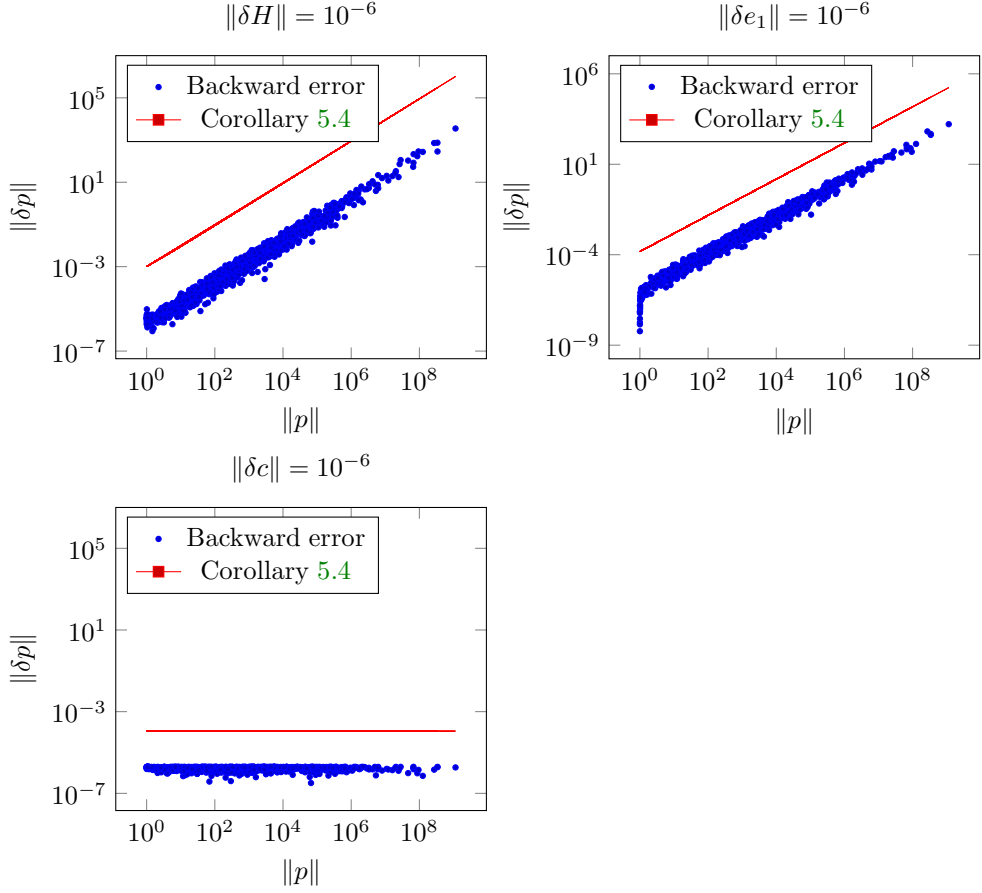
\begin{figure}
  \begin{tikzpicture}
    \begin{loglogaxis}[
        width=.5\linewidth,
        legend pos = north west,
        xlabel = {$\norm{p}$}, 
        ylabel = {$\norm{\delta p}$},
        title = {$\norm{\delta H} = 10^{-6}$},
        ymax = 1e7
      ]
      \addplot+[mark=*, only marks, blue, mark size = 1pt] table {be-h.dat};
      \addplot+ table[x index = 0, y index = 2, mark=none] {be-h.dat};
      \legend{Backward error, Corollary~\ref{cor:chebfinal}}
    \end{loglogaxis}
  \end{tikzpicture}~\begin{tikzpicture}
    \begin{loglogaxis}[
        width=.5\linewidth,
        legend pos = north west,
        xlabel = {$\norm{p}$}, 
        ylabel = {$\norm{\delta p}$},
        title = {$\norm{\delta e_1} = 10^{-6}$},
        ymax = 1e7
      ]
      \addplot+[mark=*, only marks, blue, mark size = 1pt] table {be-u.dat};
      \addplot+ table[x index = 0, y index = 2, mark=none] {be-u.dat};
      \legend{Backward error, Corollary~\ref{cor:chebfinal}}
    \end{loglogaxis}
  \end{tikzpicture} \\
  \begin{tikzpicture}
    \begin{loglogaxis}[
        width=.5\linewidth,
        legend pos = north west,
        xlabel = {$\norm{p}$}, 
        ylabel = {$\norm{\delta p}$},
        title = {$\norm{\delta c} = 10^{-6}$}, 
        ymax = 1e7
      ]
      \addplot+[mark=*, only marks, blue, mark size = 1pt] table {be-v.dat};
      \addplot+ table[x index = 0, y index = 2, mark=none] {be-v.dat};
      \legend{Backward error, Corollary~\ref{cor:chebfinal}}
    \end{loglogaxis}
  \end{tikzpicture}

  \caption{Experimental validation of the bounds from 
  Corollary~\ref{cor:chebfinal} on random Chebyshev polynomials 
  with unbalanced coefficients and degree $5$. The 
  dependency of the error on $\delta c$ for perturbations 
  $\delta H$ and $\delta e_1$ is clearly visible, whereas 
  the perturbations on $c$ are not influenced by the norm 
  of the polynomial coefficients. }
  \label{fig:cheb-examples}
\end{figure}

The bounds from Corollary~\ref{cor:chebfinal} are rather descriptive 
of the impact of the perturbations; however, we find that for 
larger degrees the quadratic terms in $n$ tend to be 
pessimistic and are rarely visible in practice. On the contrary, 
the dependency in $c$ is encountered in generic cases. 

\subsection{Proof of Lemma~\ref{lem:boundcheb}} \label{sec:prooflem}

Bounding the constant $M$ of Lemma~\ref{lem:boundcheb} requires 
to give a lower bound to
the pairwise distance between the roots of the Chebyshev polynomial
of the first kind 
of degree $n$, denoted by $r_1, \ldots, r_n$, and the ones 
of the second kind of degree $n - 1$, denoted by $\rho_1, \ldots, \rho_{n-1}$
extended with $\pm 1$ as $\rho_0$ and $\rho_n$. 
In addition, bounding $S$ requires an upper bound to the sum of their 
inverses. To obtain such results, we exploit the fact that these quantities
are explicitly known:
\begin{equation} \label{eq:rhoandr}
r_j = \cos\left(\frac{(2j+1) \pi}{2n} \right), \ \ j = 0, \ldots, n - 1, 
\qquad 
\rho_j = \cos\left(\frac{j\pi}{n} \right), \ \
j = 0, \ldots, n. 
\end{equation}
Before stating the main result, we need to state 
a few inequalities that will be key in the proof.  

\begin{lemma} \label{lem:cosdiff}
	Let $x, y$ be two positive real numbers such that $0 \leq x \leq \frac{\pi}{2}$, 
	and $0 \leq x \leq y \leq \pi$. Then, 
	\[
	\cos(x) - \cos(y) \geq \frac{4}{3\pi^2} (y^2 - x^2). 
	\]
\end{lemma}

\begin{proof}
	Let us consider two separate cases; if $y \leq \frac{\pi}{2}$, 
	we can rewrite $\cos(x) - \cos(y)$ as 
	\[
	\cos(x) - \cos(y) = 2 \sin\left(
	\frac{x + y}{2}
	\right) \sin\left(
	\frac{y - x}{2}
	\right) \geq 
	\frac{2}{\pi^2}(y^2 - x^2), 
	\]
	where we used that $\sin(z) \geq \frac{2}{\pi}z$ for $z \in [0, \pi/2]$, 
	and the fact that both $\frac{x+y}{2}$ and $\frac{y-x}{2}$ lie 
	in this interval. Then, we may consider $\frac{\pi}{2} \leq y \leq \pi$. 
	In this case, the condition $y \geq x$ is trivially satisfied, so
	it can be ignored. Then, we note that $\cos(z) \geq 1 - \frac{2}{\pi}z$
	for $z \in [0, \frac{\pi}{2}]$, and $\cos(z) \leq 1 - \frac{2}{\pi}z$
	if $z \in [\frac{\pi}{2}, \pi]$. Hence, 
	\[
	\cos(x) - \cos(y) \geq \left(1 - \frac{2}{\pi} x\right) - \left(1 - \frac{2}{\pi} y\right)
	= \frac{2}{\pi}(y - x), \qquad 
	\begin{cases}
	0 \leq x \leq \frac{\pi}{2} \\
	\frac{\pi}{2} \leq y \leq \pi \\
	\end{cases}. 
	\]
	Under these assumptions, we also have $(y+x) \leq \frac{3}{2}\pi$, so we can
	conclude that 
	\[
	\cos(x) - \cos(y) \geq \frac{2}{\pi}(y - x) = 
	\frac{2}{\pi} \frac{y^2 - x^2}{y + x} \geq 
	\frac{4}{3\pi^2} (y^2 - x^2). 
	\]
	Combining the inequalities obtained in the different parts of the
	domain yields the final result. 
\end{proof}

\begin{lemma} \label{lem:boundseries}
	Let $m \geq 1, n \geq 0$ be positive integers. Then, 
	\[
	S_1(m) := \sum_{j = 1}^{m-1} \frac{1}{m^2 - j^2} \leq \frac{1}{3} \qquad 
	S_2(m) := \sum_{j = m+1}^{n} \frac{1}{j^2 - m^2}\leq \frac{3}{4}. 
	\]
\end{lemma}

\begin{proof}
	The inequality for $S_2(m)$ can be obtained extending the summation
	to infinity and then performing a change
	of variable: 
	\begin{align*}
	\sum_{j = m+1}^{n} \frac{1}{j^2 - m^2} &\leq \sum_{j = m+1}^{\infty} \frac{1}{j^2 - m^2} = \sum_{j = 1}^{\infty} \frac{1}{(j+m)^2 - m^2} \\
	&= \sum_{j = 1}^{\infty} 
	\frac{1}{j^2 + 2mj} \leq \sum_{j = 1}^{\infty} \frac{1}{j^2 + 2j} = \frac{3}{4}, 
	\end{align*}
	where the last equality can be obtained proving, e.g., by induction, that 
	the partial sums up to $N$
	of the above series are equal to $(3N^2 + 5N) / (4N^2 + 12N + 8)$. 
	Taking the limit for $N \to \infty$ yields the desired result. 
	
	For the first inequality, we note that the summand is an increasing function
	in $j$, and therefore we can bound the summation with the 
	integral\footnote{The explicit form of the integral can be obtained using the 
		known primitive of $\frac{1}{m^2 - x^2}$ in terms of the hyperbolic
		arcotangent, and then using the expression of the latter by means
		of logarithms. The derivation is elementary but tedious, so it
		has been omitted. }
	\begin{align*}
	\sum_{j = 1}^{m - 1} \frac{1}{m^2 - j^2} &=
	\frac{1}{2m - 1} + \sum_{j = 1}^{m - 2} \frac{1}{m^2 - j^2} \leq 
	\frac{1}{2m - 1} + \int_0^{m-1} \frac{dx}{m^2 - x^2} \\
	&= \frac{1}{2m - 1} + \frac{\log(2m - 1)}{2m} =: F(m). 
	\end{align*}
	Note that the term $\frac{1}{2m-1}$ has been removed from the integral
	to avoid the singularity at $x = m$. 
	We now show that $F(m)$ is decreasing, and therefore it is sufficient
	to evaluate it at a certain $m$ to obtain bounds for all $m' > m$. To
	this aim, we compute 
	\[
	F'(m) = \frac{-2}{(2m-1)^2} + \frac{1}{m^2} \left(\frac{m}{2m-1} - \frac{\log(2m-1)}{2}\right)=
	- \frac{1}{(2m-1)^2m} - \frac{\log(2m-1)}{2m^2},
	\]
	and 
	it is immediate to verify that $F'(m) < 0$ for $m \geq 1$.
	We then substitute\footnote{The choice of $m = 6$ is motivated by
		the fact that the bound is not sharp for small values of 
		$m$, so we only use it 
		for the elements $m \geq 6$, and we check the others by 
		a direct computation.} $m = 6$, and we have 
	\[
	\sum_{j = 1}^{m-1} \frac{1}{m^2 - j^2} \leq \frac{1}{11} + 
	\frac{\log(11)}{12} \leq 0.3, \qquad 
	m \geq 6. 
	\]
	A direct inspection shows that $S_1(2) = \frac{1}{3}$, 
	and $S_1(m) \leq \frac{1}{3}$ for $m \in \{ 1, 3, 4, 5\}$.
	Therefore, 
	we conclude that $S_1(m) \leq \frac{1}{3}$ for any $m \geq 1$. 
\end{proof}

\begin{lemma} \label{lem:boundMCheb}
	Let $r_1, \ldots, r_n$ be the roots of $T_n(x)$, and 
	$\rho_j$ defined as in \eqref{eq:rhoandr}, and assume 
	$n \geq 2$. Then, if we define the function
	\[
	f_m(x) = \frac{1}{|x - r_m|}, 
	\]
	we have that $f_m(\rho_j) \leq 3 n^2$, for any $j = 0,\ldots, n$. 
\end{lemma}

\begin{proof}
	Recall that, in view of \eqref{eq:rhoandr}, $\rho_{j+1} \leq r_j \leq \rho_j$. Therefore, we only need to test the bound for 
	$j \in \{ m, m + 1 \}$. Let us consider $j = m$ first. We have:
	\[
	f_m(\rho_m) = \frac{1}{\rho_m - r_m} =
	\frac{1}{\cos\left(\frac{2m}{2n}\pi\right) - \cos\left(\frac{2m+1}{2n}\pi\right)}.  
	\]
	Assume that $\frac{2m}{2n}\pi \leq \pi/2$; this is not restrictive 
	thanks to the symmetry of the problem. Indeed, one can use the 
	change of variable $\theta \mapsto \pi - \theta$, and reduce
	to the cases considered below. 
	
	Then, 
	using Lemma~\ref{lem:cosdiff} to give a lower bound to the
	denominator we obtain
	\[
	f_m(\rho_m) \leq 
	\frac{3 \pi^2}{
		4\left(\frac{(2m+1)^2 \pi^2}{4n^2} - \frac{(2m)^2 \pi^2}{4n^2}\right)
	} = 
	\frac{3n^2}{4m + 1} \leq 3 n^2,  
	\]
	since $m \geq 0$. The case $j = m + 1$ is completely analogous. 
\end{proof}

The previous result gives a bound for the quantity $M$ of Lemma~\ref{lem:boundcheb} --- it is now necessary to consider the
summation of $\frac{1}{|r_m - \rho_j|}$, in order to bound $S$. 

\begin{lemma} \label{lem:boundSCheb}
	Let $r_1, \ldots, r_n$ be the roots of $T_n(x)$, 
	$\rho_j$ defined as in \eqref{eq:rhoandr}. If we define the function
	\[
	g(x) = \sum_{j = 1}^{n} \frac{1}{|x - r_j|}, 
	\]
	then $g(\rho_m) \leq 5 n^2$, for any $m = 0,\ldots, n$. 
\end{lemma}

\begin{proof}
	As a preliminary reduction,
	note that it is sufficient to prove the claim under the
	assumption that $\rho_m \in [0, 1]$, which is equivalent to 
	$\frac{2m}{2n} \pi \leq \frac{\pi}{2}$. Indeed, both the sets
	of $r_m$ and $\rho_m$ are symmetric with respect to the
	imaginary axis, and therefore $g(\rho_m) = g(-\rho_m)$. 
	
	We now rewrite the summation to remove the absolute
	values, recalling that $r_{m+1} \leq \rho_m \leq r_m$: 
	\[
	g(\rho_m) = 
	\underbrace{\sum_{j = 1}^{m - 1} \frac{1}{r_j - \rho_m}}_{g_1(m)} + 
	\underbrace{\sum_{j = m + 1}^{n} \frac{1}{\rho_m - r_j}}_{g_2(m)} + 
	f_{m}(\rho_m), 
	\]
	where $f_m(x)$ is defined according to 
	Lemma~\ref{lem:boundMCheb}. The last term can be bounded by $3 n^2$. Let us consider $g_1(m)$, for which we
	can write, using the same arguments of Lemma~\ref{lem:boundMCheb}, and 
	noting that $r_j = \cos(\frac{2j+1}{2n}\pi)$ are such that 
	$\frac{2j+1}{2n}\pi \leq \frac{2m}{2n} \pi \leq \frac{\pi}{2}$, 
	\begin{align*}
	g_1(m) &= \sum_{j = 1}^{m - 1} \frac{1}{r_j - \rho_m} = 
	\sum_{j = 1}^{m - 1} \frac{1}{\cos\left(\frac{(2j+1)\pi}{2n}\right) - \cos\left(\frac{2m\pi}{2n}\right)}  \leq 
	\sum_{j = 1}^{m - 1}
	\frac{3\pi^2}{4 \left[ (\frac{4m^2\pi^2}{4n^2}) - \frac{(2j+1)^2 \pi^2}{4n^2}\right]} \\ &\leq 3n^2 \sum_{j = 1}^{m-1} \frac{1}{4m^2 - (2j+1)^2} 
	\leq 3n^2 \sum_{j = 1}^{2m-1} \frac{1}{(2m)^2 - j^2} \leq n^2. 
	\end{align*}
	
	The result concerning $g_2(m)$ can be proven by 
	following similar steps. 
	\begin{align*}
	g_2(m) &= \sum_{j = m + 1}^{n} \frac{1}{\rho_m - r_j} 
	= \sum_{j = m + 1}^{n} \frac{1}{\cos(\frac{2m\pi}{2n}) - \cos(\frac{2j+1}{2n}\pi)} \\
	&\leq \frac{3n^2}{4}\sum_{j = m + 1}^{n} \frac{1}{j^2-m^2} \leq 
	\frac{9}{16}n^2. 
	\end{align*}
	where once again we used Lemma~\ref{lem:cosdiff} since 
	$\frac{m}{n} \pi \leq \frac\pi 2$, and then applied 
	Lemma~\ref{lem:boundseries}. Combining the bounds yields
	$g(m) \leq (3 + 1 + \frac{9}{16})n^2 \leq 5n^2$.
\end{proof}

\subsection{The case of Jacobi polynomials}

A natural extension of the approach described in Section~\ref{sec:chebyshev}
is to provide explicit constants for Theorem~\ref{thm:structuredbe} 
for Jacobi polynomials $P_k^{(\alpha, \beta)}(x)$, which 
are orthogonal with respect to the scalar product:
\[
\langle p, q \rangle := \int_{-1}^{-1} p(x) q(x) (1 - x)^{\alpha} (1 + x)^\beta\ dx. 
\]
The usual normalization for Jacobi polynomials is to impose that 
\[
P_k^{(\alpha, \beta)}(1) = \binom{k + \alpha}{k}.
\]
Note that this choice, when $\alpha = \beta = -\frac{1}{2}$, provides
a scaled version of the Chebyshev polynomials of the first kind, and when 
$\alpha = \beta = \frac{1}{2}$ of the ones of the second kind. In particular,
Jacobi polynomials with this scaling are orthogonal but not orthonormal, 
and we have:
\begin{equation} \label{eq:jacnorm}
\norm{P_k^{(\alpha, \beta)}}^2 = \int_{-1}^1
P_k^{(\alpha,\beta)}(x)^2 (1-x)^\alpha (1+x)^\beta\ dx = 
\frac{2^{\alpha+\beta+1}}{2k + \alpha + \beta + 1}
\frac{\Gamma(k+\alpha+1)\Gamma(k+\beta+1)}{\Gamma(k+\alpha+\beta+1)\Gamma(k+1)}. 
\end{equation}  
The recursion coefficients for Jacobi polynomials are given by (see \cite[Section~22]{abramowitz1965handbook}):
\begin{align*}
\alpha_{k} &= \frac{(2k + \alpha + \beta)(2k + \alpha + \beta - 1)}{2k(k+\alpha + \beta)} &
\beta_{k} &= \frac{(\alpha^2 - \beta^2)(2k + \alpha + \beta - 1)}{2k(k+\alpha + \beta)(2k + \alpha + \beta - 2)} \\
\gamma_{k} &= \frac{(k + \alpha - 1)(k+\beta - 1)(2k + \alpha + \beta)}{k(k+\alpha + \beta)(2k+\alpha+\beta-2)} 
\end{align*}
Hence, using the construction and the symmetrization procedure 
as in Section~\ref{sec:comrade-intro}, 
we have that
\begin{equation} \label{eq:jacobilin}
C = \begin{bmatrix}
b_{n} & c_{n-1}  \\
c_{n-1}   & b_{n-2} & \ddots \\
& \ddots & \ddots & c_1 \\
& & c_1 & b_1 \\
\end{bmatrix} - \tilde \cratio^{-1} e_1 \begin{bmatrix}
d_1 p_{n-1} & \dots & d_n p_0 
\end{bmatrix}, 
\end{equation}
where: 
\begin{align}
b_k &= \frac{\beta^2 - \alpha^2}{(2k + \alpha + \beta)(2k + \alpha + \beta - 2)} & 
c_k &=  \frac{2}{2k+\alpha+\beta}\sqrt{
	\frac{k(k+\alpha)(k+\beta)(k+\alpha+\beta)}{(2k+\alpha+\beta+1)(2k+\alpha+\beta-1)}
}. \label{eq:jacdn}
\end{align}
and 
\begin{equation} \label{eq:dkjacobi}
d_k = \sqrt{\frac{\Gamma(\alpha+k)\Gamma(\beta+k)\Gamma(\alpha+\beta+2)}{(k-1)!(2k+\alpha+\beta+1)\Gamma(\alpha+1)\Gamma(\beta+1)\Gamma(\alpha+\beta+k)}}, 
\end{equation}
and we set $d_0 = 1$ as described in Section~\ref{sec:comrade-intro}. 
We observe that $d_k = \mathcal O(k^{-\frac 12})$ for large $k$; if one was to perform the scaling of the basis numerically, this would yield the asymptotic conditioning of the task. For the degrees that are typically of practical interest, this behaviour is mild, and the scaling of the problem to get a structured matrix can be used 
without significantly altering the conditioning of the problem. 

The following lemma that will be used in the proof of Lemma~\ref{lem:ljacobi}, 
which provides the analogue result of 
Lemma~\ref{lem:boundcheb} for Jacobi polynomials. 

\begin{lemma} \label{lem:coeffs-jac-1pmx}
	Let $P_{n-1}^{(\alpha+1, \beta+1)}(x)$ the Jacobi polynomial of degree $n$, 
	with $\alpha,\beta \geq \frac{1}{2}$. Then, if 
	the coefficients $f_j$ satisfy
	\[
	(1 \pm x) P_{n-1}(x)^{(\alpha+1, \beta+1)} = \sum_{j = 0}^{n} f_j P_j^{(\alpha,\beta)}(x), 
	\]
	then $|f_j| \leq 6$. 
\end{lemma}

\begin{proof}
	We first consider the case with $(1 + x) P^{(\alpha+1, \beta+1)}_{n-1}(x)$. 
	We report the following relations among Jacobi polynomials, which can be
	found in \cite[Section~22.7]{abramowitz1965handbook}. We have:
	\begin{align}
	(1 + x) P_{n-1}^{(\alpha+1, \beta+1)}(x) &= 
	a_n P_{n-1}^{(\alpha+1, \beta)} + b_n P_{n}^{(\alpha+1, \beta)},  \label{eq:jac:1} \\
	P_n^{(\alpha+1,\beta)}(x) &= c_n P_n^{(\alpha,\beta)}(x) + d_n P_{n-1}^{(\alpha+1,\beta)}(x) \label{eq:jac:2}
	\end{align}
	where $a_n = \frac{2(n+\beta)}{2n+\alpha+\beta+1}$ and 
	$b_n = \frac{2n}{2n+\alpha+\beta+1}$, $c_n = \frac{2n+\alpha+\beta+1}{n+\alpha+\beta+1}$ and 
	$d_n = \frac{n+\beta}{n+\alpha+\beta+1}$. 
	We note that the repeated application of \eqref{eq:jac:2} yields the following:
	\[
	P_n^{(\alpha+1, \beta)}(x) = c_n P_n^{(\alpha,\beta)}(x) + d_n c_{n-1} P_{n-1}^{(\alpha,\beta)}(x) + d_n d_{n-1} c_{n-2} P_{n-2}^{(\alpha,\beta)}(x)
	+ \ldots + d_n \cdots d_1 c_0. 
	\]
	Combining this observation with \eqref{eq:jac:1} finally yields 
	\[
	f_j := \begin{cases}
	b_n c_n & \text{if } j = n \\
	(b_n d_n + a_n) c_j \prod_{s = j+1}^{n-1} d_s & 0 \leq j \leq n-1 \\
	\end{cases}
	\]
	Thanks to our assumption that $\alpha,\beta \geq \frac 12$, we have
	that $|d_j| \leq 1$, and in particular this implies that 
	$f_j \leq c_j (|a_j| + b_j)$. Since $1 \leq c_j \leq 2$, 
	$b_j \leq 1$, and $|a_j| \leq 2$, and we conclude that $|f_j| \leq 6$. 
	
	The proof for $(1 - x) P_{n-1}^{(\alpha+1, \beta+1)}(x)$ 
	is similar so we omit it. 
\end{proof}

\begin{lemma}\label{lem:ljacobi} 
	Consider the 
	nodes $\rho_0 = -1, \rho_n = 1$, and 
	$\rho_j$ the roots of $P_{n-1}^{(\alpha+1, \beta+1)}$ for $j = 1, \ldots, n-1$.
	Moreover, let $\hat L$ be the matrix defined as in Theorem~\ref{thm:structuredbe} choosing 
	the nodes as above and $\{ \phi_j \}$ the Jacobi polynomials $P_n^{(\alpha, \beta)}$.  Then, 
	\begin{align*}
	\norm{\hat L}_\infty \leq C_n^{(\alpha,\beta)} &:= 12 + (n-1) \max_j\left|\frac{w_{j-1}}{1 - x_{j-1}^2}\right| \frac{2n + \alpha + \beta + 1}{2^{\alpha+\beta+1}} \binom{\alpha+\beta+n-1}{\max\{ \alpha, \beta \}}. 
	\end{align*}
	where $w_j$ and $x_j$ are the integration weights and nodes associated with
	the orthogonal polynomial $P_{n-1}^{(\alpha+1,\beta+1)}(x)$. 
\end{lemma}

\begin{proof}
	The proof follows the same strategy and uses the same notation
	of the one given for Chebyshev
	polynomials of the first kind. 		
	We have that 
	\[
	\norm{P_{i-1}^{(\alpha, \beta)}(x)}^2 \cdot \hat L_{ij} = \int_{-1}^1 \ell_{j-1}(x) P^{(\alpha,\beta)}_{i-1}(x)(1-x)^\alpha (1+x)^\beta\ dx, \qquad 
	i,j = 1, \ldots, n + 1.
	\]
	
	If $2 \leq j \leq n$, then $\ell_{j-1}(x)$ is divisible
	by $(1 - x)^2$, since it vanishes at $\pm 1$. 
	Therefore, for $1 \leq j \leq n - 1$, we can
	define the degree $n - 2$ polynomial $q_j(x) := \ell_j(x) / (1 - x^2)$
	and rewrite the formula as follows: 
	\[
	\hat L_{ij} = \frac{1}{\norm{P_{i-1}^{(\alpha, \beta)}(x)}^2} \cdot 
	\int_{-1}^1 q_{j-1}(x) P^{(\alpha,\beta)}_{i-1}(x)(1-x)^{\alpha+1} (1+x)^{\beta+1}\ dx, \qquad 
	2 \leq j \leq n. 
	\]
	Since $\deg(q_{j-1}(x) P_{i-1}^{(\alpha,\beta)}(x)) = n + i - 3 \leq  2n - 3$, 
	because we are assuming $i \leq n$, we can
	integrate the above exactly using the Jacobi-Gauss quadrature formula
	associated with the orthogonal polynomials $P_n^{(\alpha+1,\beta+1)}$
	which yields 
	\[
	\norm{P_{i-1}^{(\alpha, \beta)}(x)}^2 \cdot \hat L_{ij} = 
	\sum_{s = 1}^{n - 1} \frac{w_s}{1 - x_s^2} \ell_{j-1}(x_s) P^{(\alpha,\beta)}_{i-1}(x_s) = 
	\frac{w_{j-1}}{1 - x_{j-1}^2}P^{(\alpha,\beta)}_{i-1}(x_{j-1}). 
	\]
	Hence, we have that 
	\begin{align*}
	|\hat L_{ij}| &\leq \max_j\left|\frac{w_{j-1}}{1 - x_{j-1}^2}\right| \cdot 
	\binom{\max\{\alpha, \beta\} + i - 1}{i - 1} \cdot \frac{2i + \alpha + \beta - 1}{2^{\alpha+\beta+1}} \frac{\Gamma(i+\alpha+\beta) \Gamma(i)}{\Gamma(i+\alpha)\Gamma(i+\beta)} \\
	&=  \max_j\left|\frac{w_{j-1}}{1 - x_{j-1}^2}\right| \frac{2i + \alpha + \beta + 1}{2^{\alpha+\beta+1}} \binom{\alpha+\beta+i-1}{\max\{ \alpha, \beta \}}.
	\end{align*}
	It remains to consider the case $j \in \{ 1, n + 1 \}$. 
	We can consider $j=n+1$ first, which is associated with 
	$\ell_{n}(x)$. Since $\ell_{n}(x)$ has as roots the zeros of $P_{n-1}^{(\alpha+1, \beta+1)}(x)$ and $-1$, we can write
	it as $\ell_{n}(x) = \gamma (1 + x) P_{n-1}^{(\alpha+1, \beta+1)}(x)$ up to a scaling factor $\gamma$. The latter
	can be determined imposing $\ell_{n}(\rho_{n}) = \ell_{n}(1) = 1$ which yields
	$\gamma = \frac{\Gamma(\alpha+1)\Gamma(n)}{2\Gamma(\alpha+n)}$ since 
	$P_{n-1}^{(\alpha+1, \beta+1)}(x)(1) = \frac{\Gamma(\alpha+n)}{\Gamma(\alpha+1)\Gamma(n)}$.
	Similarly, we can show that 
	$\ell_{0}(x) = (-1)^n \frac{\Gamma(\beta+1)\Gamma(n)}{2\Gamma(\beta+n)}(1-x)P_{n-1}^{(\alpha+1, \beta+1)}(x)$. 
	
	In addition, we may write
	\[
	(1 + x) P_{n-1}^{(\alpha+1, \beta+1)}(x) = \sum_{j = 0}^{n} f_j P_{j}^{(\alpha, \beta)}(x),  \qquad 
	(1 - x) P_{n-1}^{(\alpha+1, \beta+1)}(x) = \sum_{j = 0}^{n} g_j P_{j}^{(\alpha, \beta)}(x). 
	\]
	where $|f_j|, |g_j| \leq 6 $ in view of Lemma~\ref{lem:coeffs-jac-1pmx}. 
	Hence, we can conclude that 
	$
	|\hat L_{i1}| + |\hat L_{i,n+1}| \leq 12
	$ and 
	therefore 
	\begin{align*}
	\norm{\hat L}_\infty &\leq 12 + (n-1) \max_j\left|\frac{w_{j-1}}{1 - x_{j-1}^2}\right| \frac{2n + \alpha + \beta + 1}{2^{\alpha+\beta+1}} \binom{\alpha+\beta+n-1}{\max\{ \alpha, \beta \}}.
	\end{align*}
\end{proof}

In fact, we cannot directly use Lemma~\ref{lem:ljacobi}, as we are working with the  scaled basis $d_{n-i+1}^{-1} P_{i-1}^{(\alpha,\beta)}$. In other words, we actually need a bound on $\| D \hat L \|_\infty $, $D$ being the diagonal scaling matrix $D = \mathrm{diag}(d_1,\dots,d_n)$. However, this is readily obtained as $\| D \hat L \|_\infty \leq \| D\| \| \hat L \|_\infty$ with $\|D\| = \max_{1 \leq i \leq n} d_i.$

\begin{remark}
	We note that the constant $C_n^{(\alpha,\beta)}$ involves the quantity
	$\mu^{(\alpha,\beta)}_n := \max_j\left|\frac{w_{j-1}}{1 - x_{j-1}^2}\right|$. 
	Observe that $\mu^{(-\frac 12,-\frac 12)}_n = \frac{\pi}{n}$, 
	and this fact is used in the proof of Lemma~\ref{lem:norml-chebyshev}. For other
	Jacobi polynomials, numerical experiments suggest that, at least if $\alpha = \beta$,	then $\mu^{(\alpha,\beta)}_n \approx \frac{\pi}{n+\alpha+\frac 12}$.
	We are not aware of a proof of this conjecture; some asymptotic results in
	this direction can be found in \cite{opsomer2018asymptotics}. 
\end{remark}

In order to provide the final result for Jacobi polynomials, 
we need the analogue of Lemma~\ref{lem:boundcheb} that is stated 
for Chebyshev polynomials. 

\begin{lemma} \label{lem:boundjac}
	For Jacobi polynomials $P_n^{(\alpha,\beta)}$, 
	with the notation of Lemma~\ref{lem:pointwisebound}, and $\xi = \rho_j$ as 
	defined in Theorem~\ref{thm:pert}, there exist two moderate 
	constants $\eta_M$ and $\eta_S$, depending on $\alpha,\beta$, such that
	\[
	M \leq \eta_M n^2, \qquad 
	S \leq \eta_S n^3. 
	\]
\end{lemma}

\begin{proof}
  In view of the Frenzen-Wong formula \cite{gatteschi1985zeros} 
  we may write the roots of $P_n^{(\alpha, \beta)}$ as 
  $\cos(\theta_{n,k})$, where 
  \[
    \theta_{n,k} = t_k + \frac{1}{N^2} \left[
      \left(\alpha^2 - \frac 14\right) \frac{1 - t_k \cot(t_k)}{2 t_k}
      - \frac{\alpha^2 - \beta^2}{4} \tan\left(\frac{t_k}{2}\right)
    \right] + \mathcal O(n^{-3}),
  \]
  where $t_k := \frac{j_{\alpha,k}}{N}$,
  $N := n + \frac{\alpha+\beta+1}{2}$, and $j_{\alpha,k}$ 
  are the positive roots of the Bessel function $J_{\alpha}(x)$.
  We now estimate the 
  distance of consecutive roots, by writing:
  \[
    \theta_{n,k+1} - \theta_{n,k} = 
    t_{k+1} - t_k + \frac{1}{N^2} h(t_k, t_{k+1}) + 
    \mathcal O(n^{-3}),
  \]
  where $h(\cdot, \cdot)$ 
  collects the terms in front of $\frac{1}{N^2}$ in the 
  difference. 
  It is known that the roots $j_{\alpha,k}$
  of the Bessel function $J_{\alpha}(x)$
  are simple and asymptotically distributed such that 
  $j_{\alpha,k+1} - j_{\alpha,k} \sim \pi$ for $k \to \infty$, 
  and the smallest root is strictly positive. Hence, we observe that 
  we can give an inclusion for $t_k$ of the form 
  \[
    t_k \in \left[\frac{C_{\min}}{n}, \pi - \frac{C_{\max}}{n}\right], \qquad 
    C_{\min}, C_{\max} > 0, 
  \]
  and the constants above only depend on $\alpha, \beta$,
  and not on $n$. In addition, since the roots are well separated 
  for $k \to \infty$, we may set 
  $\gamma_{\alpha} := \inf_{k} |j_{\alpha,k+1} - j_{\alpha,k}| > 0$. 

  We now note that, to bound the separation, it is sufficient 
  to consider $k = 1, \ldots, \lceil \frac n2 \rceil$; for the 
  other roots, we may just swap the role of 
  $\alpha, \beta$ and apply the same argument. For such 
  $k$, $t_k$ belongs to (assuming $n \geq 2$) 
  the interval $[\frac{C_{\min}}{n}, \frac{3\pi}{4}]$, 
  and therefore it is immediate to check that the 
  coefficient of the $\frac{1}{N^2}$ 
  term in the Frenzen-Wong formula is uniformly bounded. Hence, 
  for $n \to \infty$, we have 
  \[
    \theta_{n,k+1} - \theta_{n,k} \geq \frac{2\max\{\gamma_\alpha, \gamma_\beta\}}{2n + \alpha + \beta + 1} + \mathcal O(n^{-2}). 
  \]
  Evaluating the cosine at these angles yields that, for large 
  $n$, 
  \[
    |\cos(\theta_{n,k+1}) - \cos(\theta_{n,k})|^{-1} \sim \mathcal O(n^2). 
  \]
  Hence, there exists a constant $\eta_{M}$ that uniformly bounds 
  the above quantity for all $n$, which allows to derive the 
  first bound of the Lemma. For the second, it is sufficient to 
  sum all these bounds over all $k' \neq k$. Note that, by construction, 
  the constants $\eta_M, \eta_S$ do not depend on $n$, but 
  only on $\alpha, \beta$. 
\end{proof}


We remark that we doubt that this bound is optimal for $\eta_S$: 
we conjecture that a clever analysis of the
bounds would lead, using similar techniques 
of the ones in Lemma~\ref{lem:boundcheb}, to control the growth of 
$S$ quadratically in $n$. We leave the analysis
 of this conjecture as an open problem.

Combining Lemma~\ref{lem:ljacobi} with Lemma~\ref{lem:boundjac} yields the 
following result. 
\begin{corollary}
	Let $C = H - \cratio^{-1} e_1 c^T$ the linearization for a polynomial
	$p(x)$ expressed in the scaled 
	Jacobi basis $d_{n-j}^{-1} P_j^{(\alpha, \beta)}$ 
	for $j = 0, \ldots, n$ where $d_j$ are defined in \eqref{eq:dkjacobi}. Consider perturbations 
	$\norm{\delta H}_2 \leq \epsilon_H$, $\norm{\delta e_1} \leq \epsilon_1$, 
	and $\norm{\delta c} \leq \epsilon_c$. Then, the matrix $C + \delta C := H + \delta H - \tilde \cratio^{-1} (e_1 + \delta e_1) (c + \delta c)^T$ linearizes 
	the polynomial \[
	p(x) + \delta p(x) := \sum_{j = 0}^n (p_j + \delta p_j) d_{n-j}^{-1} P^{(\alpha,\beta)}_j(x), 
	\]
	where 
	\[
    \begin{aligned}
	|\delta p_j| &\leq \eta_D \cdot 
	\hat C_n^{(\alpha, \beta)} \left(
	\eta_M \norm{c}_2 \epsilon_1 n^2 + \tilde \cratio \epsilon_c n^{\frac 52}  + 
	(\eta_S n^3 + (\eta_M+ \eta_S)\tilde \cratio {n}^{\frac 72}) \norm{c}_2 \epsilon_H
	\right) \\
  &+ \mathcal O(\epsilon_H^2 + \epsilon_1^2 + \epsilon_c^2). 
    \end{aligned}
	\]
	where $\hat C_{n}^{\alpha,\beta} = C_n^{(\alpha, \beta)}  \binom{\max\{\alpha,\beta\} + n}{n}$ with 
  $C_n^{(\alpha,\beta)}$ defined as in Lemma~\ref{lem:ljacobi} and 
	$\eta_D := \frac{\max_{j}{d_j}}{\min_{j}{d_j}}$. 
\end{corollary}

\begin{proof}
	The result follows by applying Theorem~\ref{thm:structuredbe} together
	with Lemma~\ref{lem:ljacobi} and \ref{lem:boundjac}, and using the fact
	that \[
	|d_{n-j}^{-1} P_j^{(\alpha,\beta)}(x)| \leq \frac{\binom{\max\{\alpha,\beta\} + n}{n}}{\min_{j}{d_j}}.
	\]
\end{proof}

\section{Conclusions}
\label{sec:conc}

We have presented a backward error analysis applicable to computing roots of 
polynomials through structured QR solvers. The results cover the cases where the
error has the same normal-plus-rank-one structure of the confederate matrix, 
and the backward errors on the various parts have different magnitudes. 

This often happens in practice when the structure is exploited, 
as in the algorithm presented in \cite{aurentz2015fast}
for the monomial case. We have provided an alternative derivation that
recovers the results of the stability
analysis in \cite{aurentz2015fast}.

These results have then been extended to the Chebyshev and Jacobi basis, with explicit
bounds provided. This suggests the requirements that a QR-based rootfinder in these
bases needs to have to obtain a stable rootfinding algorithm. 

Some related topics might be subject to future investigation. For instance,
an algorithm for symmetric-plus-rank-one matrices arising from polynomial
rootfinding satisfying the proposed stability constraints does not exist yet. 
Our hope is that this paper suggests research directions to 
develop
one. 

Another research line stemming from this analysis is extending the results to
the case of matrix
polynomials. Polynomial eigenvalue problems can be solved using unitary-plus-low-rank
solvers in the monomial basis \cite{aurentz2016fast}, or symmetric-plus-low-rank ones
for more general bases \cite{eidelman2008efficient}.
However, the use of the determinant to recover the linearized polynomial is 
not applicable in the matrix polynomial  setting, and other more involved questions such
as the accurate (stable) computation 
of the eigenvectors are of interest as well.  

\bibliographystyle{siamplain}
\bibliography{structured_errors}

\end{document}